\documentclass[reqno]{amsart}

\usepackage[T1]{fontenc}
\usepackage{lmodern}
\usepackage{amssymb,mathtools}
\usepackage{microtype}
\usepackage{aliascnt}

\usepackage[colorlinks=true,linkcolor=blue,citecolor=blue,urlcolor=blue]{hyperref}
\usepackage[nameinlink,noabbrev]{cleveref}

\mathtoolsset{centercolon,showonlyrefs}

\theoremstyle{plain}
\newtheorem{theorem}{Theorem}

\newaliascnt{proposition}{theorem}

\aliascntresetthe{proposition}

\newaliascnt{lemma}{theorem}
\newtheorem{lemma}[lemma]{Lemma}
\aliascntresetthe{lemma}

\newaliascnt{corollary}{theorem}
\newtheorem{corollary}[corollary]{Corollary}
\aliascntresetthe{corollary}

\theoremstyle{definition}
\newaliascnt{definition}{theorem}
\newtheorem{definition}[definition]{Definition}
\aliascntresetthe{definition}

\newaliascnt{example}{theorem}
\newtheorem{example}[example]{Example}
\aliascntresetthe{example}

\theoremstyle{remark}
\newaliascnt{remark}{theorem}
\newtheorem{remark}[remark]{Remark}
\aliascntresetthe{remark}

\crefname{theorem}{Theorem}{Theorems}
\Crefname{theorem}{Theorem}{Theorems}
\crefname{proposition}{Proposition}{Propositions}
\Crefname{proposition}{Proposition}{Propositions}
\crefname{lemma}{Lemma}{Lemmas}
\Crefname{lemma}{Lemma}{Lemmas}
\crefname{corollary}{Corollary}{Corollaries}
\Crefname{corollary}{Corollary}{Corollaries}
\crefname{definition}{Definition}{Definitions}
\Crefname{definition}{Definition}{Definitions}
\crefname{example}{Example}{Examples}
\Crefname{example}{Example}{Examples}
\crefname{remark}{Remark}{Remarks}
\Crefname{remark}{Remark}{Remarks}

\crefname{appendix}{Appendix}{Appendices}
\Crefname{appendix}{Appendix}{Appendices}

\newcommand{\R}{\mathbb{R}}
\newcommand{\df}{\mathrm{d}}
\newcommand{\dd}{\mathop{}\!\df}

\DeclarePairedDelimiter\abs{\lvert}{\rvert}

\newcommand{\W}{\mathcal{W}}
\newcommand{\A}{\mathcal{A}}

\newcommand{\HIp}{{H_{\mathrm{I}}^{+}}}
\newcommand{\HIIp}{{H_{\mathrm{II}}^{+}}}
\newcommand{\HIIIA}{{H_{\mathrm{III}}^{A}}}

\title{The Independence-Preserving Property and Planar Web Geometry}
\author{Yoshihiro Gyotoku}
\address{Graduate School of Mathematical Sciences, The University of Tokyo,
3-8-1 Komaba, Meguro-ku, Tokyo 153-8914, Japan}
\email{gyotoku@ms.u-tokyo.ac.jp}
\email{yoshihirogyotoku0131@gmail.com}
\thanks{
This work was supported by JSPS KAKENHI Grant Nos.~24KJ0933 and 24K21515, by the FoPM WINGS Program at the University of Tokyo, and in part by the National Science Centre, Poland, under Project
No.~2023/51/B/ST1/01535.
}

\begin{document}

\begin{abstract}
This paper introduces planar web geometry into the study of characterization problems for probability laws arising from independence preservation.

Let $F(x,y) = (u, v)$ be a real-analytic diffeomorphism defined on an open domain in $\R^2$.
Let $X$ and $Y$ be independent random variables such that $(X, Y)$ takes values in the domain of $F$, and let $\mu_X$ and $\mu_Y$ be their laws.
Set $(U, V) = F(X, Y)$ and denote the laws of $U$ and $V$ by $\mu_U$ and $\mu_V$.
The map $F$ is said to preserve the independence of $(\mu_X, \mu_Y, \mu_U, \mu_V)$ if $\mu_U\otimes\mu_V = F_\#(\mu_X\otimes\mu_Y)$, which is equivalent to the independence of $U$ and $V$.
This condition is highly restrictive and can therefore be used to characterize the laws for which it holds.

For measures with positive continuous densities, the map preserves independence for such product measures if and only if their logarithmic densities solve the corresponding inhomogeneous Abelian functional equation.
The solution space of logarithmic densities is identified with an affine space modeled on the space of Abelian relations, whose dimension is at most three according to Bol's rank bound under non-singularity assumptions of the associated web.

This framework provides unified proofs of characterization theorems, including classical results such as the Kac–Bernstein and Lukacs theorems and the Matsumoto–Yor property, as well as recent results for the quadrirational Yang–Baxter maps $\HIp$, $\HIIp$, and $\HIIIA$, under the stated regularity assumptions.
The same method also produces further examples, including a class generalizing transformations studied by Koudou and Vallois; every map in this class admits at least a two-parameter family of preserved product measures.

The independence-preserving property of a map is preserved by coordinatewise one-variable reparametrizations, which also preserve the associated web.
Thus the web associated with a map is a natural invariant under the corresponding equivalence relation of independence-preserving maps.
Together with the local algebraizability of nonsingular planar $4$-webs of maximal rank, this correspondence suggests a route toward a complete classification of maps preserving the independence of a three-parameter family of measures.
\end{abstract}

\maketitle

\section{Introduction}
\label{sec:introduction}

Characterizations of probability laws by independence preservation
properties form a classical subject.
Let $X$ and $Y$ be independent random variables.
The Kac--Bernstein theorem
\cite{Kac1939,Bernstein1941} characterizes Gaussian laws
by the independence of $X+Y$ and $X-Y$.
The Lukacs theorem \cite{Lukacs1955} characterizes gamma laws
by the independence of $X+Y$ and $X/Y$ for positive $X$ and $Y$.
The Matsumoto--Yor property \cite{MatsumotoYor2001} concerns an
independence relation involving the generalized inverse Gaussian
(GIG) and gamma laws, and a corresponding converse characterization was established by Letac and Weso\l{}owski
\cite{LetacWesolowski2000}.
Further independence properties of Matsumoto--Yor type were studied
by Koudou and Vallois \cite{KoudouVallois2012}; see also
\cite{KoudouLey2014}.
Such problems are called independence-characterization problems.

Motivated by studies of invariant measures of discrete integrable systems,
Sasada and Uozumi \cite{SasadaUozumi2024} showed that quadrirational Yang--Baxter maps $\HIp$, $\HIIp$ and $\HIIIA$ from \cite{PSTV2010} preserve independence of certain probability measures on $(0,\infty)^2$.
The corresponding laws were subsequently characterized: the generalized beta-prime family associated with
$\HIp$ in \cite{KLPW2025}, the Kummer family associated with
$\HIIp$ in \cite{KoudouWesolowski2025}, and the GIG family
associated with $\HIIIA$ in \cite{LetacWesolowski2024}.
These characterizations were proved separately by arguments
adapted to the individual transformations.

The present paper introduces the theory of planar web geometry and Abelian relations \cite{BlaschkeBol1938,ChernGriffiths1978,Chern1982,AkivisGoldberg2000} to this circle of characterization problems, uncovers a common underlying structure and provides unified proofs under additional regularity assumptions; see \cref{cor:HI-densities,cor:HII-densities,cor:HIII-densities,rem:classical-properties}.

For each $T\in\{X, Y, U, V\}$, let $I_T$ be an open interval of $\R$.
Let $F(x,y)=(u,v)$ be a real-analytic diffeomorphism from $I_X\times I_Y$ to $I_U\times I_V$.
Let $X$ and $Y$ be independent random variables taking values in $I_X$ and $I_Y$ respectively, and let $(U, V) = F(X, Y)$.
The map $F$ is said to preserve the independence of $X$ and $Y$ if $U$ and $V$ are independent.
This happens if and only if
\[
\mu_U \otimes \mu_V = F_\#(\mu_X \otimes \mu_Y),
\]
where $\mu_T$ denotes the probability distribution of $T$.
Then $F$ is also said to preserve the independence of $(\mu_X, \mu_Y, \mu_U, \mu_V)$.
This terminology is extended to general measures by the same identity.

For each $T\in\{X, Y, U, V\}$, let $\mu^0_T$ and $\mu_T$ be $\sigma$-finite Borel measures on $I_T$, both equivalent to Lebesgue measure. 
Assume that $\dd\mu_T/\!\dd\mu^0_T$ admits a positive continuous representative, and let $\varphi_T = \log(\dd\mu_T/\!\dd\mu^0_T)$.
Let
\[
\W_F=(x,y,u(x, y),v(x, y))
\]
be the planar $4$-web associated with $F$ and assume it is nonsingular.
Suppose that $F$ preserves the independence of $(\mu^0_X, \mu^0_Y, \mu^0_U, \mu^0_V)$.
\Cref{thm:characterization}, the main theorem, states that $F$ preserves the independence of $(\mu_X, \mu_Y, \mu_U, \mu_V)$
if and only if
\[
(\varphi_X,\varphi_Y,-\varphi_U,-\varphi_V)
\in\A(\W_F),
\]
where $\A(\W_F)$ is the vector space of Abelian relations for $\W_F$.
Thus the equation governing the logarithmic density-ratio $\varphi_T$ is precisely the Abelian functional equation associated with $\W_F$.
Consequently, $\dim\A(\W_F)$ is the number of free parameters of the family of measures modulo componentwise positive rescaling within the stated regularity class for which $F$ preserves independence.

For a nonsingular real-analytic web, measurable Abelian relations have
real-analytic representatives \cite{Pirio2006,PirioThesis}, and Bol's theorem gives
\[
\dim\A(\W_F)\leq 3.
\]
In particular, three linearly independent Abelian relations form a basis of $\A(\W_F)$ and determine the family of measures satisfying regularity assumptions whose independence is preserved by $F$.
This framework indeed recovers the known results, including classical results associated with the Kac--Bernstein and Lukacs theorems and the Matsumoto--Yor property and recent ones for quadrirational maps.

The use of web geometry is not limited to recovering known examples.
\Cref{ex:rank-two-class} shows maps of the form $F(x, y) = (\phi(x+y), \psi(x)-\phi(x+y))$, which generalizes transformations considered by Koudou and Vallois \cite{KoudouVallois2012}, exhibit at least
a two-parameter family of preserved product measures.
Also, in \cref{ex:HI-on-other-domains} the map $\HIp$ restricted to other domains are considered, with characterizations of three-parameters family of preserved product measures.
These examples show both how further examples can be constructed and how
their independence properties can be proved without a separate
functional-equation argument for each transformation.

There is also a structural reason for attaching a web to the independence-preservation problem.
Separate one-variable reparametrizations of $x,y,u$ and $v$ by diffeomorphisms $a,b,c$ and $d$ applied to each varibale preserve the independence-preserving property of a map.
For the associated web, the same operation replaces $(x,y,u,v)$ by $(a(x),b(y),c(u(x, y)),d(v(x, y)))$. 
Since composing a function with a one-variable diffeomorphism does not change its level curve, the associated web remain unchanged. 
Hence the associated web for a map is a natural geometric invariant under these coordinatewise changes of variables.
See \cref{rem:coordinate-equivalence} for details.

The maximal-rank case has an additional classical property:
every nonsingular planar $4$-web of maximal rank is locally
algebraizable
\cite{BlaschkeBol1938,ChernGriffiths1978,Chern1982}.
Together with the preceding correspondence, this makes it natural to seek a complete classification of real-analytic diffeomorphisms that preserve independence of measures in a three-parameter family.


\section{A web-geometric framework}\label{sec:framework}

\subsection{Planar webs and Abelian relations}\label{sec:web-input}

\begin{definition}\label{def:web}
    Let $d$ be a positive integer and $\Omega\subset\R^2$ be an open connected set.
    A \emph{nonsingular planar $d$-web} on $\Omega$ is a $d$-tuple $\W=(h_i)_{i=1}^d$ of real-analytic functions such that $\dd h_i\neq 0$ and $\dd h_i\wedge\dd h_j\neq 0$ everywhere on $\Omega$ whenever $i\neq j$.
    \footnote{
    Geometrically, each function \(h_i\) represents the family of its level curves $\mathcal{F}(h_i) = \{\{(x,y)\in\Omega \mid h_i(x,y)=t\} \mid t\in h_i(\Omega) \}$. 
    Since for a one-variable diffeomorphism $\rho_i$ one has $\mathcal{F}(h_i) = \mathcal{F}(\rho_i\circ h_i)$, the web $(\rho_i\circ h_i)_{i=1}^d$ is identified with $(h_i)_{i=1}^d$.
    }

    A \emph{measurable Abelian relation} for $\W$ is a $d$-tuples $(\Phi_i)_{i=1}^d$ of real-valued measurable functions $\Phi_i$ on $h_i(\Omega)$ satisfying
    \begin{align}\label{eq:Afe}
        \Phi_1(h_1(x, y)) + \cdots + \Phi_d(h_d(x, y)) = (\text{constant})
    \end{align}
    for every $(x, y)\in\Omega$.
    This equation is called an \emph{Abelian functional equation}.
    For differentiable components, this is equivalent to
    \[
        \Phi_1'(h_1(x, y)) \dd h_1 + \cdots + \Phi_d'(h_d(x, y)) \dd h_d = 0.
    \]
    Continuous and real-analytic Abelian relations are defined analogously.
    Two measurable Abelian relations are identified if their corresponding components differ by additive constants.
    The real vector space of real-analytic Abelian relations modulo this equivalence is denoted by $\A(\W)$, and its dimension is called the rank of $\W$.
\end{definition}

The following regularity result is \cite[Proposition~1.2.1]{PirioThesis} applied to the connected set $\Omega$ under the assumption that $\W$ is nonsingular everywhere on $\Omega$.
See also \cite[Theorem~A]{Pirio2006}, which treats the case where all $h_i$ are rational.
\begin{lemma}\label{lem:pirio}
    Let $\W$ be a nonsingular planar web on $\Omega$.
    Every measurable Abelian relation for $\W$ has a real-analytic representative.
\end{lemma}

The classical rank bound will be used; see \cite{Bol1932} and \cite[Proposition~4]{Pirio2006}.
\begin{lemma}\label{lem:bol}
    The rank of a nonsingular planar $d$-web is at most $(d-1)(d-2)/2$.
\end{lemma}

\subsection{Product reference measures and the density equation}

Throughout the paper, identities between measures are understood up to multiplication by positive constants.
For measures $\mu$ and $\mu'$, the notation $\mu\propto\mu'$ means $\mu=c\mu'$ for some $c>0$.
Accordingly, logarithmic densities are understood up to additive constants.
A probability measure is considered as a representative of an equivalence class of proportional finite measures.

\begin{definition}\label{def:ip}
    For each $T\in\{X,Y,U,V\}$, let $I_T\subset\R$ be an open interval and $\mu_T$ be a $\sigma$-finite Borel measure on $I_T$.
    The map $F \colon I_X\times I_Y \to I_U\times I_V$ is said to preserve the independence of $(\mu_X,\mu_Y,\mu_U,\mu_V)$ if
    \begin{equation}\label{eq:ip-measure}
        F_\#(\mu_X\otimes\mu_Y)\propto\mu_U\otimes\mu_V.
    \end{equation}
\end{definition}
This terminology originates from the fact that if the measures $\mu_X$, $\mu_Y$ are probability measures and $X$ and $Y$ are independent random variables whose distributions are $\mu_X$ and $\mu_Y$ respectively, then $F$ preserves independence of $(\mu_X,\mu_Y,\mu_U,\mu_V)$ if and only if the random variables $U$ and $V$ defined by $(U, V) = F(X, Y)$ are independent and their distributions are $\mu_U$ and $\mu_V$ respectively.

\begin{remark}\label{rem:coordinate-equivalence}
    Let \(a,b,c,d\) be one-variable real-analytic diffeomorphisms between
    intervals, and set $\widetilde F=(c\times d)\circ F\circ(a\times b)^{-1}$.
    Then
    \[
    F_\#(\mu_X\otimes\mu_Y)\propto\mu_U\otimes\mu_V
    \iff
    \widetilde F_\#\big(a_\#\mu_X\otimes b_\#\mu_Y\big)
    \propto \big(c_\#\mu_U\otimes d_\#\mu_V\big).
    \]
    For probability measures, this correspondence has a direct interpretation:
    \[
    (U,V)=F(X,Y) \iff (c(U),d(V))=\widetilde F(a(X),b(Y)).
    \]
    Applying separate functions to independent random variables preserves
    independence, and the invertibility of \(a,b,c,d\) makes this correspondence reversible. 
    Thus the displayed equivalence gives a bijection between the
    product-measure quadruples whose independence is preserved by \(F\) and
    those whose independence is preserved by \(\widetilde F\). 
    It is therefore natural to regard \(F\) and \(\widetilde F\) as equivalent in an independence-characterization problem.
    
    On the other hand, the web accoiated to $\widetilde F$ pulled back to the original domain is $(a(x), b(y), c(u(x, y)), d(v(x, y)))$, which is identified with $\W_F=(x,y,u(x,y),v(x,y))$.
    Therefore, the web is a natural geometric invariant under the prescribed equivalence.
    
    This fact supports the idea of introducing the theory of web geometry into the analysis of the independence-characterization problem, as seen in the main theorem.
\end{remark}

The following theorem translates the independence-characterization problem into inhomogeneous Abelian functional equation\footnote{Abelian functional equation \eqref{eq:Afe} with nontrivial function on the right hand side} \eqref{eq:psi}, which once one solution is fixed, reduces to homogeneous Abelian functional equation \eqref{eq:varphi}.

\begin{theorem}\label{thm:characterization}
    For each $T\in\{X,Y,U,V\}$, let $I_T$ be an open interval of $\R$.
    Let $\mu^0_T$ and $\mu_T$ be $\sigma$-finite Borel measures on $I_T$, both equivalent to Lebesgue measure. 
    Assume that the Radon--Nikodym derivative $\dd\mu_T/\!\dd\mu^0_T$ admits a positive continuous representative, and let $\varphi_T$ denote its logarithm.
    
    Let $F \colon I_X\times I_Y \to I_U\times I_V$, $F(x,y)=(u,v)$ be a real-analytic diffeomorphism, and assume that the planar $4$-web $\W_F=(x,y,u(x,y),v(x,y))$ on $I_X\times I_Y$ is nonsingular.
    
    Suppose that $F$ preserves the independence of $(\mu^0_X,\mu^0_Y,\mu^0_U,\mu^0_V)$.
    Then $F$ preserves the independence of $(\mu_X,\mu_Y,\mu_U,\mu_V)$ if and only if
    \begin{equation}\label{eq:abelian-density}
        (\varphi_X,\varphi_Y,-\varphi_U,-\varphi_V)\in\A(\W_F).
    \end{equation}
\end{theorem}

\begin{proof}
    Assume first that $F$ preserves the independence of $(\mu_X,\mu_Y,\mu_U,\mu_V)$.
    For each $T\in\{X,Y,U,V\}$, as $\mu_T$ and $\mu^0_T$ are assumed to be equivalent to Lebesgue measure, their densities are almost surely positive.
    Let $\psi_T$ and $\psi^0_T$ denote the logarithmic density modulo additive constant, which satisfy
    \begin{align}
        \mu_T(\dd t) &\propto \exp\big(\psi_T(t)\big) \dd t, &
        \mu^0_T(\dd t) &\propto \exp\big(\psi^0_T(t)\big) \dd t, &
        \varphi_T(t) &= \psi_T(t) - \psi^0_T(t).
    \end{align}
    Let $J_F$ denote the Jacobian determinant of $F$.
    Since \eqref{eq:ip-measure} holds for $(\mu_X,\mu_Y,\mu_U,\mu_V)$ and $(\mu^0_X,\mu^0_Y,\mu^0_U,\mu^0_V)$, taking logarithms of the change-of-variables formula yields
    \begin{align}
        \psi_X(x) + \psi_Y(y)
        - \psi_U \bigl(u(x,y)\bigr) - \psi_V \bigl(v(x,y)\bigr)
        &= \log |J_F| + (\text{constant}),
        \label{eq:psi}
        \\
        \psi^0_X(x) + \psi^0_Y(y)
        - \psi^0_U \bigl(u(x,y)\bigr) - \psi^0_V \bigl(v(x,y)\bigr)
        &= \log |J_F| + (\text{constant})
        \label{eq:psi0}
    \end{align}
    for Lebesgue almost every $(x,y)\in I_X\times I_Y$.
    Subtracting \eqref{eq:psi0} from \eqref{eq:psi}, one obtains
    \begin{equation}\label{eq:varphi}
        \varphi_X(x)+\varphi_Y(y)
        -\varphi_U\bigl(u(x,y)\bigr)-\varphi_V\bigl(v(x,y)\bigr)
        = (\text{constant})
    \end{equation}
    Lebesgue almost everywhere.
    Since both sides are continuous, \eqref{eq:varphi} holds everywhere.
    Thus $(\varphi_X,\varphi_Y,-\varphi_U,-\varphi_V)$ is a measurable Abelian relation for $\W_F$.
    By \cref{lem:pirio}, it has a real-analytic representative and belongs to $\A(\W_F)$, hence \eqref{eq:abelian-density} holds.
    
    Conversely, \eqref{eq:abelian-density} implies \eqref{eq:varphi} holds everywhere.
    Since $F$ preserves $(\mu^0_X,\mu^0_Y,\mu^0_U,\mu^0_V)$, \eqref{eq:psi0} holds Lebesgue almost everywhere.
    Adding \eqref{eq:varphi} and \eqref{eq:psi0} yields \eqref{eq:psi} Lebesgue almost everywhere, which is again equivalent to \eqref{eq:ip-measure}.
\end{proof}

The following corollary shows that if one has a three-parameter family of measures whose independence is preserved by a two-dimensional diffeomorphism, then the characterization theorem is immediate under the regularity assumptions.

\begin{corollary}\label{thm:characterization-rank3}
    Under the assumptions of \cref{thm:characterization}, assume moreover that three linearly independent Abelian relations 
    \begin{align}
        \big(\Phi_i^X(x), \Phi_i^Y(y), -\Phi_i^U(u), -\Phi_i^V(v)\big) \in \A(\W_F) \quad (i = 1, 2, 3)
    \end{align}
    are given.
    Then $F$ preserves the independence of $(\mu_X, \mu_Y, \mu_U, \mu_V)$ if and only if there exist real parameters $p$, $q$ and $r$ satisfying
    \begin{align}\label{eq:phi-is-superposition}
        \mu_T(\dd t) \propto \exp\big( p \Phi_1^T(t) + q \Phi_2^T(t) + r \Phi_3^T(t) \big) \, \mu^0_T(\dd t)
    \end{align}
    for each $T\in\{X,Y,U,V\}$.
\end{corollary}

\begin{proof}
    \Cref{thm:characterization} identifies the solution space for $(\varphi_X,\varphi_Y,-\varphi_U,-\varphi_V)$ with $\A(\W_F)$, whose dimension is at most $(4-1)(4-2)/2 = 3$, according to \cref{lem:bol}.
    Therefore the exhibition of three linearly independent Abelian relations is sufficient to form a basis, and any element of $\A(\W_F)$ is a linear combination of them.
\end{proof}

\begin{remark}
    The maximal-rank case $\dim\A(\W_F) = 3$ is geometrically distinguished: nonsingular planar $4$-webs of maximal rank are locally algebraizable \cite{BlaschkeBol1938,ChernGriffiths1978,Chern1982}.
    Combined with the idea in \cref{rem:coordinate-equivalence}, this supports the author's conjecture that the structure theory of maximal-rank webs can yield a complete classification, up to coordinatewise reparametrization, of real-analytic diffeomorphisms admitting a three-parameter family of preserved product measures.
\end{remark}

\section{Applications}\label{sec:application}
\subsection{Quadrirational Yang--Baxter maps}

The characterization theorems in \cite{KLPW2025,KoudouWesolowski2025,LetacWesolowski2024} for quadrirational Yang--Baxter maps can be recovered from \cref{thm:characterization-rank3} under additional assumptions.

\begin{definition}
    Fix $\alpha\not=\beta\in\mathbb{CP}^1$.
    Define $\HIp$, $\HIIp$ and $\HIIIA$ by
    \begin{alignat}{3}
        \label{eq:HI}
        \HIp(x, y)
        &=                                                                      
        \Big\lparen
            \frac{y}{\alpha}\,\frac{\beta+\alpha x+\beta y+\alpha \beta x y}{1+x+y+\beta x y}
            &,\quad& \frac{x}{\beta}\,\frac{\alpha+\alpha x+\beta y+\alpha \beta x y}{1+x+y+\alpha x y}
        &&\Big\rparen,
        \\
        \label{eq:HII}
        \HIIp(x, y)
        &=
        \Big\lparen
            \frac{y}{\alpha}\,\frac{\beta+ \alpha x+\beta y}{1+x+y}
            &,\quad& \frac{x}{\beta}\,\frac{\alpha+ \alpha x+\beta y}{1+x+y}
        &&\Big\rparen,
        \\
        \label{eq:HIII}
        \HIIIA(x, y)
        &=
        \Big\lparen
            \frac{y}{\alpha}\,\frac{\alpha x+\beta y}{x+y}
            &,\quad& \frac{x}{\beta}\,\frac{\alpha x+\beta y}{x+y}
        &&\Big\rparen.
    \end{alignat}
\end{definition}

Direct calculation shows that each map is a birational involution of $(\mathbb{CP}^1)^2$.
If $\alpha, \beta \in (0, \infty)$, the displayed formulas have positive numerators and denominators on $(0,\infty)^2$; hence, the maps can be restricted to involutions of $(0,\infty)^2$.

\begin{lemma}\label{lem:jacobian-haar}
    For each $F\in \{\HIp|_{(0,\infty)^2},\HIIp|_{(0,\infty)^2},\HIIIA|_{(0,\infty)^2}\}$, the web $\W_F$ is nonsingular and $F$ preserves $(\dd x/x, \dd y/y, \dd u/u, \dd v/v)$.
\end{lemma}

\begin{proof}
    Let $J_F$ denote the Jacobian determinant of $F$.
    Since $\dd x\wedge\dd y\neq 0$ everywhere,
    \begin{align}
        \dd x\wedge \dd u &= u_y\,\dd x\wedge\dd y, &
        \dd y\wedge \dd u &= -u_x\,\dd x\wedge\dd y, \\
        \dd x\wedge \dd v &= v_y\,\dd x\wedge\dd y, &
        \dd y\wedge \dd v &= -v_x\,\dd x\wedge\dd y, &
        \dd u\wedge \dd v &= J_F\,\dd x\wedge\dd y
    \end{align}
    and none of $u_x$, $u_y$, $v_x$, $v_y$ or $J_F$ vanishes on $(0, \infty)^2$, the web is nonsingular.
    Direct calculation shows $J_F=-uv/(xy)$, hence $F$ preserves $(\dd x/x, \dd y/y, \dd u/u, \dd v/v)$.
\end{proof}

\begin{corollary}\label{cor:HI-densities}
Let $(\mu_X,\mu_Y,\mu_U,\mu_V)$ be probability measures on $(0, \infty)$ having positive continuous density.
Then $\HIp|_{(0,\infty)^2}$ preserves their independence if and only if there exist $p,q,r\in\R$ with $\max\{p, 0\} < \min\{q, r\}$ such that
\begin{alignat*}{3}
    \mu_X(\dd x)
        &\propto (1 + \alpha x)^{p}\Big(\frac{1}{1+x}\Big)^q\Big(\frac{\alpha x}{1+\alpha x}\Big)^r\frac{\dd x}{x} &
        &\propto x^{r-1}(1+x)^{-q}(1+\alpha x)^{p-r}\,\dd x,\\
    \mu_Y(\dd y)
        &\propto (1 + \beta y)^{p}\Big(\frac{\beta y}{1+\beta y}\Big)^q\Big(\frac{1}{1+y}\Big)^r\frac{\dd y}{y} &
        &\propto y^{q-1}(1+y)^{-r}(1+\beta y)^{p-q}\,\dd y,\\
    \mu_U(\dd u)
        &\propto (1 + \alpha u)^{p}\Big(\frac{\alpha u}{1+\alpha u}\Big)^q\Big(\frac{1}{1+u}\Big)^r\frac{\dd u}{u} &
        &\propto u^{q-1}(1+u)^{-r}(1+\alpha u)^{p-q}\,\dd u,\\
    \mu_V(\dd v)
        &\propto (1 + \beta v)^{p}\Big(\frac{1}{1+v}\Big)^q\Big(\frac{\beta v}{1+\beta v}\Big)^r\frac{\dd v}{v} &
        &\propto v^{r-1}(1+v)^{-q}(1+\beta v)^{p-r}\,\dd v.
\end{alignat*}
\end{corollary}

\begin{proof}
    \Cref{lem:jacobian-haar} allows one to take $\mu^0_T(\dd t) = \dd t/t$ for each $T\in\{X,Y,U,V\}$.
    The following three quadruples are linearly independent elements of $\A(\W_{\HIp|_{(0,\infty)^2}})$:
    \begin{alignat*}{7}
        &\big\lparen&
            \log(1 + \alpha x)&,\,\,&
            \log(1 + \beta y)&,\,\,&
            - \log(1 + \alpha u)&,\,\,&
            - \log(1 + \beta v)&
        &\big\rparen,
        \\
        &\Big\lparen&
            \log\frac{1}{1+x}&,\,\,&
            \log\frac{\beta y}{1+\beta y}&,\,\,&
            - \log\frac{\alpha u}{1+\alpha u}&,\,\,&
            - \log\frac{1}{1+v}&
        &\Big\rparen,
        \\
        &\Big\lparen&
            \log\frac{\alpha x}{1+\alpha x}&,\,\,&
            \log\frac{1}{1+y}&,\,\,&
            - \log\frac{1}{1+u}&,\,\,&
            - \log\frac{\beta v}{1+\beta v}&
        &\Big\rparen.
    \end{alignat*}
    Then \cref{thm:characterization-rank3} applies, and the measures are finite if and only if $\max\{p,0\}<\min\{q,r\}$.
\end{proof}

\begin{corollary}\label{cor:HII-densities}
Let $(\mu_X,\mu_Y,\mu_U,\mu_V)$ be probability measures on $(0, \infty)$ having positive continuous density.
Then $\HIIp|_{(0,\infty)^2}$ preserves their independence if and only if there exist $p,q,r\in\R$ with $p,q,r>0$ such that
\begin{alignat*}{2}
    \mu_X(\dd x)
        &\propto (e^{-\alpha x})^p x^q (1+x)^{-r} \dd x/x
        & &\propto x^{q-1}(1+x)^{-r}e^{-p\alpha x}\,\dd x,\\
    \mu_Y(\dd y)
        &\propto (e^{-\beta y})^p (1+y)^{-q} y^r \dd y/y
        & &\propto y^{r-1}(1+y)^{-q}e^{-p\beta y}\,\dd y,\\
    \mu_U(\dd u)
        &\propto (e^{-\alpha u})^p (1+u)^{-q} u^r \dd u/u
        & &\propto u^{r-1}(1+u)^{-q}e^{-p\alpha u}\,\dd u,\\
    \mu_V(\dd v)
        &\propto (e^{-\beta v})^p v^q (1+v)^{-r} \dd v/v
        & &\propto v^{q-1}(1+v)^{-r}e^{-p\beta v}\,\dd v.
\end{alignat*}
\end{corollary}

\begin{proof}
    \Cref{lem:jacobian-haar} allows one to take $\mu^0_T(\dd t) = \dd t/t$ for each $T\in\{X,Y,U,V\}$.
    The following three quadruples are linearly independent elements of $\A(\W_{\HIIp|_{(0,\infty)^2}})$:
    \begin{alignat*}{7}
        &\lparen&
            -\alpha x&,\,\,&
            - \beta y&,\,\,&
            \alpha u&,\,\,&
            \beta v
        &&\rparen,
        \\
        &\lparen&
            \log x&,\,\,&
            - \log(1+y)&,\,\,&
            \log(1+u)&,\,\,&
            - \log v
        &&\rparen,
        \\
        &\lparen&
            - \log(1+x)&,\,\,&
            \log y&,\,\,&
            - \log u&,\,\,&
            \log(1+v)
        &&\rparen.
    \end{alignat*}
    Then \cref{thm:characterization-rank3} applies, and the measures are finite if and only if $p,q,r>0$.
\end{proof}

\begin{corollary}\label{cor:HIII-densities}
Let $(\mu_X,\mu_Y,\mu_U,\mu_V)$ be probability measures on $(0, \infty)$ having positive continuous density.
Then $\HIIIA|_{(0,\infty)^2}$ preserves their independence if and only if there exist $p,q,r\in\R$ with $p,q>0$ such that
\begin{alignat*}{2}
    \mu_X(\dd x)
        &\propto (e^{-\alpha x})^p (e^{-1/x})^q x^r \dd x/x
        & &\propto x^{r-1}e^{-p\alpha x-q/x}\,\dd x,\\
    \mu_Y(\dd y)
        &\propto (e^{-\beta y})^p (e^{-1/y})^q y^{-r} \dd y/y
        & &\propto y^{-r-1}e^{-p\beta y-q/y}\,\dd y,\\
    \mu_U(\dd u)
        &\propto (e^{-\alpha u})^p (e^{-1/u})^q u^{-r} \dd u/u
        & &\propto u^{-r-1}e^{-p\alpha u-q/u}\,\dd u,\\
    \mu_V(\dd v)
        &\propto (e^{-\beta v})^p (e^{-1/v})^q v^r \dd v/v
        & &\propto v^{r-1}e^{-p\beta v-q/v}\,\dd v.
\end{alignat*}
\end{corollary}

\begin{proof}
    By \cref{lem:jacobian-haar}, take $\mu^0_T(\dd t)=\dd t/t$ for each $T\in\{X,Y,U,V\}$.
    The following three quadruples are linearly independent elements of $\A(\W_{\HIIIA|_{(0,\infty)^2}})$:
    \begin{align}
        (-\alpha x,-\beta y,\alpha u,\beta v),\quad
        (-1/x,-1/y,1/u,1/v),\quad
        (\log x,-\log y,\log u,-\log v).
    \end{align}
    Then \cref{thm:characterization-rank3} applies, and the measures are finite if and only if $p,q>0$.
\end{proof}

\subsection{Further examples}

The classical results mentioned in \Cref{sec:introduction} can also be recovered from \cref{thm:characterization} under additional assumptions.

\begin{remark}\label{rem:classical-properties}
For the Kac--Bernstein theorem, consider $F(x,y)=(x+y,x-y)$ on $\R^2$ and let $\mu^0_T(\dd t)=\dd t$.
Three Abelian relations are
\begin{align}
    (x,y,-u,0),\quad
    (x,-y,0,-v),\quad
    (2x^2,2y^2,-u^2,-v^2).
\end{align}
For the Lukacs theorem, consider $F(x,y)=(x+y,x/y)$ on $(0, \infty)^2$ and let $\mu^0_T(\dd t)=\dd t/t$.
Three Abelian relations are
\begin{align}
    (x,y,-u,0), \quad
    (-\log x,\log y,0,\log v),\quad
    (0,\log y,-\log u,\log(1+v)).
\end{align}
The Matsumoto--Yor property is a special case of the $\HIIIA$; see \cite[Remark~3.4]{SasadaUozumi2024}.
\end{remark}

The converse direction of \cref{thm:characterization} can be used to construct further examples of independence-preserving maps, possibly with fewer than three parameters.
The following example provides another class of independence-preserving maps, generalizing transformations considered by Koudou and Vallois \cite{KoudouVallois2012}\footnote{The case $\phi=\psi$ with deceasing bijection $\phi$ on $(0,\infty)$ with some regularities are treated.}.

\begin{example}\label{ex:rank-two-class}
Let $\phi$ and $\psi$ be smooth bijections on the relevant intervals such that
\begin{equation}\label{eq:phi-psi-map}
    F(x,y)=(u,v)=\bigl(\phi(x+y),\,\psi(x)-\phi(x+y)\bigr)
\end{equation}
defines a real-analytic diffeomorphism between product domains and its associated web $\W_F$ is nonsingular.
\footnote{
This condition is indeed satisfied for example by $\phi(t) = e^t$, $\psi(t) = e^{t+1}$ and $I_X = \R$, $I_Y = (-\infty, 1)$ and $I_U = I_V = (0, \infty)$. 
}
Then
\[
(x,y,-\phi^{-1}(u),0),\qquad
(\psi(x),0,-u,-v)
\]
are in $\A(\W_F)$.
Since $J_F(x,y)=-\phi'(x+y)\psi'(x)$, one may take $\mu^0_X(\dd x) = \abs{\psi'(x)}\dd x$, $\mu^0_Y(\dd y) = \dd y$, $\mu^0_U(\dd u) = \dd u/\abs{\phi'(\phi^{-1}(u))}$ and $\mu^0_V(\dd v) = \dd v$ as references.
Therefore the converse direction of \cref{thm:characterization} shows that the following measures satisfy \eqref{eq:ip-measure}:
\begin{align}
    \mu_X(\dd x)&\propto \exp\bigl(-px-q\psi(x)\bigr)\abs{\psi'(x)}\dd x,
    &\mu_Y(\dd y)&\propto \exp(-py)\dd y,\\
    \mu_U(\dd u)&\propto
    \exp\bigl(-p\phi^{-1}(u)-qu\bigr)
    \frac{\dd u}{\abs{\phi'(\phi^{-1}(u))}},
    &\mu_V(\dd v)&\propto \exp(-qv)\dd v,
\end{align}
for $p,q\in\R$.
Thus $F$ preserves at least a two-parameter family of measures.
If the rank of the web $\W_F$ is not three, then these form the complete family of measures satisfying assumptions of \cref{thm:characterization} whose independence is preserved by $F$.
The author conjectures that if the rank of the web $\W_F$ is three, then $F$ is equivalent, in the sense of \cref{rem:coordinate-equivalence}, to one of the maps in the list of \cite[Theorem 2.2]{KoudouVallois2012}.
\end{example}

\begin{remark}
    Since webs and their Abelian functional equations can be considered locally, one may change the domain of one independence-preserving map to obtain another independence-preserving map.
    The characterized measures share the same density functions although the domains are different.
    See the example below.
\end{remark}

\begin{example}\label{ex:HI-on-other-domains}
    If $\alpha$ and $\beta$ are different and greater than $1$, then direct calculation shows that the formula of the map $\HIp$ also defines an involutive diffeomorphism on $\Omega\in\{(-1/\alpha, 0) \times (-1/\beta, 0), (-1, -1/\alpha) \times (-1, -1/\beta), (-\infty, -1)^2\}$.
    Similar arguments as in the proofs of \cref{lem:jacobian-haar,cor:HI-densities} with $\mu^0_T(\dd t) = \dd t/|t|$ prove that independence preservation by $\HIp|_\Omega$ characterizes the following probability measures considered on the corresponding intervals for each map:
    \begin{alignat*}{3}
        \mu_X(\dd x)
            &\propto |1 + \alpha x|^{p}\Big|\frac{1}{1+x}\Big|^q\Big|\frac{\alpha x}{1+\alpha x}\Big|^r\frac{\dd x}{|x|} &
            &\propto |x|^{r-1}|1+x|^{-q}|1+\alpha x|^{p-r}\,\dd x,\\
        \mu_Y(\dd y)
            &\propto |1 + \beta y|^{p}\Big|\frac{\beta y}{1+\beta y}\Big|^q\Big|\frac{1}{1+y}\Big|^r\frac{\dd y}{|y|} &
            &\propto |y|^{q-1}|1+y|^{-r}|1+\beta y|^{p-q}\,\dd y,\\
        \mu_U(\dd u)
            &\propto |1 + \alpha u|^{p}\Big|\frac{\alpha u}{1+\alpha u}\Big|^q\Big|\frac{1}{1+u}\Big|^r\frac{\dd u}{|u|} &
            &\propto |u|^{q-1}|1+u|^{-r}|1+\alpha u|^{p-q}\,\dd u,\\
        \mu_V(\dd v)
            &\propto |1 + \beta v|^{p}\Big|\frac{1}{1+v}\Big|^q\Big|\frac{\beta v}{1+\beta v}\Big|^r\frac{\dd v}{|v|} &
            &\propto |v|^{r-1}|1+v|^{-q}|1+\beta v|^{p-r}\,\dd v.
    \end{alignat*}
    Note that each term inside absolute values has constant sign on each domain.
    The parameter conditions for finiteness depend on the domain $\Omega$ as follows:
    \begin{alignat}{7}
        \Omega & {}={} & (-1/\alpha, 0) &\times (-1/\beta, 0)& &\implies& 0<q<p+1,\, 0<r<p+1, \\
        \Omega & {}={} & (-1, -1/\alpha) &\times (-1, -1/\beta)& &\implies& \max\{q, r\}<\min\{1, p+1\}, \\
        \Omega & {}={} & (-\infty, -1) &\times(-\infty, -1)& &\implies& p<q<1,\, p<r<1. \\
    \end{alignat} 
\end{example}

\section*{Acknowledgments}

The author is grateful to Makiko Sasada, Jacek Weso\l{}owski, Hokuto Konno, Ralph Willox and Luc Pirio for valuable discussions, comments and suggestions that improved the presentation of the paper.

\bibliographystyle{amsalpha}
\bibliography{refs}

\end{document}